\documentclass[a4paper]{article}
\usepackage{amsthm,amssymb,amsmath,enumerate,graphicx,psfrag}

\newtheorem{definition}{Definition}

\newtheorem{theorem}[definition]{Theorem}

\newtheorem{lemma}[definition]{Lemma}

\newtheorem{problem}[definition]{Problem}

\newcommand{\comment}[1]{}

\newcommand{\emtext}[1]{\text{\em #1}}

\newcommand{\del}{\ensuremath\text{\rm del}}
\newcommand{\sm}{\setminus}
\newcommand{\ind}{\ensuremath\mathcal{I}}
\newcommand{\Imax}{\ensuremath\mathcal{I^{\rm max}}}
\newcommand{\I}{\mathcal I}
\newcommand{\C}{\mathcal C}

\newcommand{\cf}{\kappa}

\title{Finite connectivity in infinite matroids}
\author{Henning Bruhn \and Paul Wollan}
\date{}

\begin{document}
\maketitle

\begin{abstract}
We introduce a connectivity function for infinite matroids with properties
similar to the connectivity function of a finite matroid, such as submodularity
and invariance under duality. As an application we use it to extend Tutte's 
linking theorem to finitary and to co-finitary matroids. 
\end{abstract}

\section{Introduction}

There seems to be a common misconception\footnote{
Compare the Wikipedia entry on matroids:
``The theory of infinite matroids is much more complicated than that of finite 
matroids and forms a subject of its own. One of the difficulties is that there 
are many reasonable and useful definitions, none of which captures all the important
aspects of finite matroid theory. For instance, it seems to be hard to have bases, 
circuits, and duality together in one notion of infinite matroids.'' (15/03/2010)
}
that a matroid on an infinite ground set 
has to sacrifice at least one of the key features of matroids: the existence of bases, or of circuits,
or duality. This is not so. As early as 1969 a model of infinite matroids, called B-matroids, was proposed
by Higgs~\cite{Higgs69b,Higgs69c,Higgs69} that turns out to possess many of the common properties associated with finite matroids.
Unfortunately, Higgs' definition and exposition were not very accessible, and although Oxley~\cite{Oxley78,Oxley92}
presented a much simpler definition and made a number of substantial contributions,
the usefulness of Higgs' notion remained somewhat obscured. To address this, it is 
shown in~\cite{infaxioms} that infinite matroids can be equivalently described by simple and concise
sets of independence, basis, circuit and closure axioms, much in the same way
as finite matroids.
See Section~\ref{sec:defs} for more details.

 
In this article, we introduce a connectivity function for infinite matroids and show
that it allows one to extend Tutte's linking theorem to at least a large class of infinite matroids.

When should we call an infinite matroid $k$-connected? For $k=2$ this is easy: 
a finite matroid is $2$-connected, or simply \emph{connected}, if every two elements
are contained in a common circuit.  This 
definition can be extended verbatim to infinite matroids. To show that such a definition gives rise to connected components
need some more work, though, as it is non-trivial to show in infinite matroids that containment
in a common circuit yields an equivalence relation. We prove this in Section~\ref{sec:cyclecon}.

For larger $k$, $k$-connectivity is defined via the connectivity function 
$\cf_M(X)=r(X)+r(E\sm X)-r(M)$, where $X$ is a subset of the ground set of a matroid $M$
and $r$ the rank function.
Then, a finite matroid is $k$-connected unless   there exists an $\ell$-separation for 
some $\ell < k$, i.e. a subset $X\subseteq E$ so that 
$\cf_M(X)\leq \ell -1$ and $|X|,|E\sm X|\geq \ell$.
In an infinite matroid, $\cf_M$ makes not much sense as all the involved ranks will
normally be infinite. In Section~\ref{sec:highcon} we give an alternative definition of $\cf_M$
that defaults to the usual connectivity function for a finite matroid but does
extend to infinite matroids. We show, in Section~\ref{sec:propcon},
that our connectivity function has some of the expected properties, such as submodularity
and invariance under duality.

In a finite matroid, 
Tutte's linking theorem allows 
the connectivity between two fixed sets to be preserved when taking minors.
To formulate this, a refinement of the connectivity function is defined as follows:
$\cf_M(X,Y)=\min\{\cf_M(U):X\subseteq U\subseteq E\sm Y\}$ for 
any two disjoint $X,Y\subseteq E(M)$.

\begin{theorem}[Tutte~\cite{TutteLinking}]
\label{TLinking}
Let $M$ be a finite 
matroid, and let $X$ and $Y$ be two disjoint subsets of $E(M)$.  Then 
there exists a partition $(C, D)$ of $E(M) \sm (X \cup Y)$ such that $\cf_{M/C-D}(X, Y) = 
\cf_M(X, Y)$.  
\end{theorem}

As the main result of this article we prove that, based on our connectivity function, 
Tutte's linking theorem extends to a large class of infinite matroids. A matroid is finitary
if all its circuits are finite, and it is co-finitary if it is the dual of a finitary matroid.

\begin{theorem}
\label{thm:linking}
Let $M$ be a finitary 
or co-finitary matroid, and let $X$ and $Y$ be two disjoint subsets of $E(M)$.  Then 
there exists a partition $(C, D)$ of $E(M) \sm (X \cup Y)$ such that $\cf_{M/C-D}(X, Y) = 
\cf_M(X, Y)$.  
\end{theorem}

We conjecture that Theorem~\ref{thm:linking} holds  for arbitrary 
matroids.


\section{Infinite matroids and their properties}\label{sec:defs}

Similar to finite matroids,  infinite matroids can be defined by a variety of equivalent sets of axioms.  
Higgs originally defined his B-matroids by giving a set of somewhat technical axioms
for the closure operator. This was improved upon by Oxley~\cite{Oxley78,Oxley92} who 
gave a far more accessible definition. We follow~\cite{infaxioms}, where simple and 
consistent sets of independence, basis, circuit and closure axioms are provided.

Let $E$ be some (possibly infinite) set, let $\mathcal I$ be a set of subsets of $E$, and denote
by $\Imax$ the $\subseteq$-maximal sets of $\I$. 
For a set $X$ and an element $x$, we abbreviate $X\cup\{x\}$ to $X+x$, and we 
write $X-x$ for $X\sm \{x\}$.
We call $M=(E,\mathcal I)$ a \emph{matroid} if the following conditions are
satisfied:
\begin{itemize}
\item[\rm (I1)] $\emptyset\in\mathcal I$.
\item[\rm (I2)] $\I$ is closed under taking subsets.
\item[\rm (I3)] For all $I\in\I\sm\Imax$ and $I'\in\Imax$ there is an $x\in I'\sm I$ such that $I+x\in\I$.\looseness=-1
\item[\rm (IM)]  Whenever $I\subseteq X\subseteq E$ and $I\in\I$, the set $\{\,I'\in\I: I\subseteq I'\subseteq X\,\}$ 
has a maximal element.
\end{itemize}
As usual, any set in $\mathcal I$ is called \emph{independent},
any subset of $E$ not in $\mathcal I$ is \emph{dependent}, and any minimally independent set
is a \emph{circuit}. Any $\subseteq$-maximal set in $\I$ is a \emph{basis}.

Alternatively, we can define matroids in terms of circuit axioms. 
As for finite matroids we have that (C1) a circuit cannot be empty
and that (C2) circuits are incomparable. While the usual circuit exchange
axioms does hold in an infinite matroid it turns out to be too weak
to define a matroid. Instead, we have the following stronger version: 
\begin{itemize}
\item[\rm (C3)] Whenever $X\subseteq C\in\C$ and $(C_x:x\in X)$ is a family of elements of~$\C$ such 
that $x\in C_y\Leftrightarrow x=y$ for all $x,y\in X$, then for 
every $z \in C\sm \left( \bigcup_{x \in X} C_x\right)$ there exists an 
element $C'\in\C$ such that $z\in C'\subseteq \left(C\cup  \bigcup_{x \in X} C_x\right) \sm X$.
\end{itemize}
We will need the full strength of (C3) in this paper.
The circuit axioms are completed by (CM) which states that the subsets
of $E$ that do not contain a circuit satisfy (IM). For more details as well
as the basis and closure axioms, see~\cite{infaxioms}.

The main feature of this definition is that even on infinite ground sets
matroids have bases, circuits and minors while maintaining duality at the same time.
Indeed, most, if not all, standard properties of finite
matroids that have a rank-free description  carry over to infinite matroids. 
In particular, every dependent set contains a circuit, every independent
set is contained in a basis and duality is defined as one expects,
that is, $M^*=(E,\mathcal I^*)$ is the dual matroid of a matroid 
$M=(E,\mathcal I)$ if $B^*\subseteq E$ is a basis of $M^*$
if and only if $E\sm B^*$ is a basis of $M$. The dual of a matroid is 
always a matroid. As usual, we call independent sets of $M^*$ 
\emph{co-independent} in $M$, and similarly we will speak of
\emph{co-dependent} sets, \emph{co-circuits} and \emph{co-bases}.
Some of the facts above were already proved by Higgs~\cite{Higgs69}
using the language of B-matroids, the other facts are due to Oxley~\cite{Oxley78};
all of these can be found in a concise manner in~\cite{infaxioms}.
Below we list some further properties that we need repeatedly. 

\medskip
An important subset of matroids are the finitary matroids.
A set $\I$ of subsets of $E$ forms the independent sets of 
a \emph{finitary matroid} if in addition to (I1) and (I2), $\I$
satisfies the usual augmentation axiom for independent sets\footnote{
If $I,I'\in\I$ and $|I|<|I'|$ then there exists an $x\in I'\sm I$ so that $I+x\in\I$.}
as well as the following axiom:
\begin{itemize}
\item[\rm (I4)] $I\subseteq E$ lies in $\I$ if all its finite subsets are contained in $\I$.
\end{itemize}
Higgs~\cite{Higgs69} showed that finitary matroids are indeed matroids in our sense. 
In fact, a matroid is finitary if and only if 
every of its circuits is finite~\cite{infaxioms}.
A matroid is called \emph{co-finitary} if its dual matroid
is finitary. 

Finitary matroids occur quite naturally. For instance, the finite-cycle matroid of
a graph and any matroid based on linear independence are finitary. 
The uniform matroids, in which any subset of cardinality at most~$k\in\mathbb N$
is independent, provide another example.
The duals of finitary matroids 
will normally not be finitary -- we will see precisely when this happens 
in the next section.



\medskip
Let us continue with a number of useful but elementary properties of (infinite) matroids.
Higgs~\cite{Higgs69b} showed that every two bases have the same cardinality
if the generalised continuum hypothesis is assumed. We shall only need
a weaker statement:
\begin{lemma}$\!\!${\bf\cite{infaxioms}}
\label{lem:basediff}
If $B,B'$ are bases of a matroid  with $|B_1\sm B_2|<\infty$ then 
$|B_1\sm B_2|=|B_2\sm B_1|$.
\end{lemma}

Let $M=(E,\mathcal I)$ be a matroid, and let $X\subseteq E$.
We define the \emph{restriction of $M$ to $X$}, denoted by $M|X$,
as follows: $I\subseteq X$ is independent in $M|X$ if
and only if it is independent in $M$. We write $M-X$ for
$M|(E\sm X)$. 
We define the \emph{contraction of $M$ to $X$} by $M^*.X:=(M|X)^*$,
and we  abbreviate $M.(E\sm X)$ by $M/X$.
Both restrictions and contractions of a matroid are again a matroid,
see Oxley~\cite{Oxley92} or~\cite{infaxioms}.

\begin{lemma}[Oxley~\cite{Oxley92}]
\label{lem:contrbase}
Let $X\subseteq E$, and let $B_X\subseteq X$. Then the following are
equivalent
\begin{enumerate}[\rm (i)]
\item $B_X$ is a basis of $M.X$;
\item there exists a basis $B$ of $M-X$ so that $B_X\cup B$ is a basis of $M$; and
\item for all bases $B$ of $M-X$ it holds that $B_X\cup B$ is a basis of $M$.
\end{enumerate}
\end{lemma}
A proof of the lemma in terms of the independence axioms is contained in~\cite{infaxioms}.

If $B$ is a basis of a matroid $M$ then for any element $x$ outside
$B$ there is exactly one circuit contained in $B+ x$, the
\emph{fundamental circuit of $x$}, see Oxley~\cite{Oxley78}.
A \emph{fundamental co-circuit} is a fundamental circuit of the dual matroid.
We need two more elementary lemmas about circuits. 

\begin{lemma}\label{cir2}
Let $M$ be a matroid and $X \subseteq E(M)$ with $X \neq \emptyset$.   
If for every co-circuit $D$ of $M$ such that $D \cap X \neq \emptyset$, 
it holds that $|D \cap X| \ge 2$, then $X$ is dependent.  
If $C$ is a circuit of $M$, 
then it holds that $|D \cap C| \ge 2$ for every co-circuit $D$ such that 
$D \cap C \neq \emptyset$.
\end{lemma}

\begin{proof}
First, 
assume that $|D \cap X| \ge 2$ for every co-circuit $D$ such that $D \cap X \neq \emptyset$, but 
contrary to the claim, that $X$ is independent.  Then there exists a co-basis $B^*$ 
contained in $E(M)\sm X$.  If we fix an element $x \in X$ and 
consider the fundamental co-circuit contained in $B^*+x$, we find a co-circuit 
intersecting $X$ in exactly one element, namely $x$, a contradiction.  Thus we conclude that $X$ 
is dependent.

Now consider a circuit $C$ and assume, to reach a contradiction, that there 
exists a co-circuit $D$ such that $D \cap C = \{x\}$ 
for some element $x \in C$.  The set  
$D - x$ is co-independent, so there exists a basis $B$ of $M$ disjoint from $D - x$.  
Now, (IM) yields an independent set $I$ that is $\subseteq$-maximal among all
independent sets $J$ with $C-x\subseteq J\subseteq (C-x)\cup B$. By~(I3), $I$ is a 
basis of $M$, as otherwise there would be an element $b\in B\sm I$ so that $I+b$ is independent.
However, $I$ is disjoint from $D$, which contradicts that $D$ is 
co-dependent.  
\end{proof}



\begin{lemma}\label{cir3}
Let $M$ be a matroid, and let $X$ be any subset of $E(M)$.  
Then for every circuit $C$ 
of $M /X$, there exists a subset $X' \subseteq X$ 
such that $X' \cup C$ is a circuit of $M$.
\end{lemma}

\begin{proof}
Let $B_X$ be a base of $M|X$.  The independent sets of $M/X$ are  
all sets $I$ such that $I \cup B_X$ is independent in $M$.  Since $C$ is a circuit of $M/X$, 
it follows that $C \cup B_X$ is dependent and so it contains a circuit $C'$.  By the minimality of circuits, 
it follows that $C' \sm X = C $.  The set $X':=C' \cap X$ is as desired by the claim.
\end{proof}

\section{Connectivity}\label{sec:cyclecon}

A finite matroid is connected if and only if every two elements are 
contained in a common circuit. Clearly, this definition can be extended
verbatim to infinite matroids. It is, however, not clear anymore that 
this definition makes much sense in infinite matroids. Notably, the
fact that being in a common circuit is an equivalence relation
needs proof.
To provide that proof is the main aim of this section.

Let $M=(E,\mathcal I)$ be a fixed matroid in this section.
Define a relation $\sim$ on $E$ by: $x\sim y$ if and only 
if there is a circuit in $M$ that contains $x$ and $y$.
As for finite matroids, we say that $M$ is \emph{connected}
if $x\sim y$ for all $x,y\in E$.

\begin{lemma}
\label{lem:equiv}
$\sim$ is an equivalence relation.
\end{lemma}

The proof will require two simple facts that we note here.

\begin{lemma}
\label{propA}
If $C$ is a circuit and $X \subsetneq C$, then $C\sm X$ is a circuit in $M/X$.
\end{lemma} 
\begin{proof}
If $C \sm X$ 
is not a circuit then there exists a set $C' \subsetneq C \sm X$ such that $C'$ is a circuit of $M/X$.  
Now, Lemma~\ref{cir3} yields a set $X' \subseteq X$ such that 
$C' \cup X'$ is a circuit of $M$, and this will 
be a proper subset of $C$, a contradiction.  
\end{proof}

\begin{lemma}
\label{propB}
Let $e\in E$ be contained in a circuit of $M$, and consider $X\subseteq E-e$.
Then $e$ is contained in a circuit of $M/X$.
\end{lemma}
\begin{proof}
Let $e$ be contained in a circuit $C$ of $M$, and suppose that $e$
does not lie in any circuit of $M/X$. Then $\{e\}$ is a  a co-circuit of $M/X$,
and thus also  
a co-circuit of $M$. This, however, contradicts Lemma~\ref{cir2} since the circuit 
$C$ intersects the co-circuit 
$\{e\}$ in exactly one element.
\end{proof}

\begin{proof}[Proof of Lemma~\ref{lem:equiv}]
Symmetry and reflexivity are immediate.  To see transitivity, let $e$, $f$, and $g$ 
in $E$ be given such that $e,f$ lie in a common circuit $C_1$, and 
$f,g$ are contained in a circuit $C_2$.  
We will find a subset $X$ of the ground set such that $M/X$ contains 
a circuit containing both $e$ and $g$.  By Lemma~\ref{cir3},
this will suffice to prove the claim.

\medskip

First, we claim that without loss of generality we may assume that
\begin{equation}\label{simpleworld}
E(M)=C_1\cup C_2\emtext{ and }C_1\cap C_2=\{f\}.
\end{equation}
Indeed, as any circuit in any restriction of $M$ is still a circuit
of $M$, we may delete any element outside $C_1\cup C_2$. Moreover,
we may contract  $(C_1 \cap C_2) - f$.  
Then $(C_1 \sm C_2) + f$ is a circuit containing $e$ and $f$, and similarly, $(C_2 \sm C_1) +f$ 
is a circuit containing both $f$ and $g$ by Lemma~\ref{propA}.
Any circuit $C$ with $e,g\in C$ in $M/((C_1\cap C_2)-f)$ will extend to 
a circuit in $M$, by Lemma~\ref{cir3}.  Hence, we may assume~\eqref{simpleworld}.

Next, we attempt to contract the set $C_2 \sm \{f, g\}$.  
If $C_1$ is a circuit of $M/(C_2 \sm \{f, g\})$, 
then we can find a circuit containing both $e$ and $g$ by applying the circuit 
exchange axiom $(C2)$ to the circuit $C_1$ and the circuit $\{f, g\}$.  Thus 
we may assume that $C_1$ is not a circuit, 
but by Lemma~\ref{propB}, 
it contains a circuit $C_3$ containing the element $e$.  If the circuit $C_3$ 
also contains the element $f$, then again by the circuit exchange axiom, 
we can 
find a circuit containing both $e$ and $g$.  Therefore, we instead assume that $C_3$ does 
not contain the element $f$.  Consequently, there exists a non-empty 
set $A \subseteq C_2 \sm \{f, g\}$ such that 
$C_3 \cup A$ is a circuit of $M$.  

Contract the set $C_2 \sm (\{f, g\} \cup A)$.  We claim that the set $C_3 \cup A$ is a circuit of 
the contraction.  If not, there exist  sets $D \subseteq C_3$ and $B \subseteq A$ such that 
$D \cup B$ is a circuit of $M /( C_2 \sm (\{f, g\} \cup A))$.  
Furthermore, $D \cup B \cup X$ is a circuit of $M$ for some set 
$X \subseteq C_2 \sm (\{f, g\} \cup A)$.   This implies that 
$D = C_3$, since $D$ contains a circuit of $M/(C_2 \sm \{f, g\})$.  If $A \neq B$, we apply the circuit 
exchange axiom to the two circuits $C_3 \cup A$ and $C_3 \cup B \cup X$ to find a circuit contained 
in their union that does not contain the element $e$.  However, the existence of such a circuit is a 
contradiction.  Either it would be contained as a strict subset of $C_2$, or upon contracting 
$C_2 \sm \{f,g\}$ we would have a circuit contained as a strict subset of $C_3$.  This final contradiction 
shows that $C_3 \cup A$ is a circuit of $M/\left(C_2 \sm (A \cup \{f, g\})\right)$.  

We now consider two circuits in $M/\left(C_2 \sm (A \cup \{f, g\})\right)$.  
The first is 
$C_1':=C_3 \cup A$, which contains $e$.  
The second is $C_2':=\{f, g\} \cup A$, the remainder of $C_2$ 
after contracting $C_2 \sm (A \cup \{f, g\})$ (note Lemma~\ref{propA}).  
We have shown that in attempting to find a circuit 
containing $e$ and $g$ utilising two circuits $C_1$ containing $e$ and $C_2$ containing $g$, 
we can restrict our attention to the case when $C_2 \sm C_1$ consists of exactly two elements.  The 
argument was symmetric, so in fact we may assume that $C_1 \sm C_2$ also consists of only 
two elements.  In~\eqref{simpleworld} we observed that we may assume that $C_1$ and $C_2$ 
intersect in exactly one element.  Thus we have reduced to a matroid 
on five elements, in which it is easy to find a circuit containing both $e$ and $g$.
\end{proof}


\medskip
Let the equivalence classes of the relation $\sim$ be the \emph{connected components} of a matroid $M$.  

As an application of Lemma~\ref{lem:equiv} we shall show that
every matroid is the direct sum of its connected components. With 
a little extra effort this will allow us to re-prove a characterisation
matroids that are both finitary and co-finitary, that had been 
noted  by Las Vergnas~\cite{LasVergnas71}, 
and by Bean~\cite{Bean76} before.

Let $M_i=(E_i,\mathcal I_i)$ 
be a collection of matroids indexed by a set $I$.
We define 
the \emph{direct sum} of the $M_i$, written $\bigoplus_{i \in I} M_i$, to have ground set consisting 
of $E:=\bigcup_{i \in I} E_i$ and independent sets $\ind = \left \{\bigcup_{i \in I} J_i: J_i \in \ind_i\right\}$.  

As noted by Oxley \cite{Oxley92} for finitary matroids, it is easy to check that:
\begin{lemma}
The direct sum of matroids $M_i$ for $i \in I$ is  a matroid.
\end{lemma}
%


\begin{lemma}\label{lem:fin1}
Every matroid is the direct sum of the restrictions to its connected components.
\end{lemma}

\begin{proof}
Let $M=(E,\mathcal I)$ be a matroid. As $\sim$ is an equivalence relation,
the ground set $E$ partitions into connected components $E_i$, for some
index set $I$. Setting $M_i:=M|E_i$, we claim that $\bigoplus_{i \in I} M_i$ 
and $M$ have the same independent sets.

Clearly, if $I$ is independent in $M$, then $I\cap E_i$ is independent in $M_i$
for every $i\in I$, which implies that $I$ is independent in $\bigoplus_{i \in I} M_i$.
Conversely, consider a set $X\subseteq E$ that is dependent in $M$. 
Then, $X$ contains a circuit $C$, which, in turn, lies in $E_j$ for some $j\in I$.
Therefore, $X\cap E_j$ is dependent, implying that $X$ is dependent in 
$\bigoplus_{i \in I} M_i$ as well.
\end{proof}

We now give the characterisation of matroids that are both finitary and co-finitary.  
\begin{theorem}\label{thm:fin} 
A matroid $M$ is both finitary and co-finitary if 
and only if there exists an index set $I$ and finite matroids $M_i$ for $i \in I$ 
such that
${M} = \bigoplus_{i \in I}M_i.$
\end{theorem}

Theorem \ref{thm:fin} is a direct consequence of the following lemma,
which has previously been proved by Bean~\cite{Bean76}.
The theorem was first proved by Las Vergnas~\cite{LasVergnas71}.
Our proof is different from the proofs of Las Vergnas and of Bean.

\begin{lemma}\label{lem:fin2}
An infinite, connected matroid contains 
either an infinite circuit or an infinite co-circuit.
\end{lemma}
\begin{proof}
Assume, to reach a contradiction, that $M$ is a connected matroid with 
$|E(M)| = \infty$ such that every circuit and every co-circuit of $M$ 
is finite.  
Fix an element $e \in E(M)$ and let $C_1, C_2, C_3, \dots$ be an infinite 
sequence of distinct circuits each containing $e$.  Let $M' = M|
\left( \bigcup_{i=1}^\infty C_i\right)$ be 
the restriction of $M$ to the union of all the circuits $C_i$.  Note that $M'$
contains a countable number of elements by our assumption that every circuit is finite.  Let $e_1, e_2, 
\dots$ be an enumeration of $E(M')$ such that $e_1 = e$.  We now recursively 
define an infinite  set $\mathcal C_i$ 
of circuits and a finite set $X_i$ for $i \ge 1$. 
Let $\mathcal{C}_1 = \{C_i : i \ge 1\}$ 
and $X_1 = \{e_1\}$.  
Assuming $\mathcal C_i$ and $X_i$ are defined for $i = 1, 2, \dots, k$, we define $\mathcal{C}_{k+1}$ 
as follows. If infinitely many circuits in $\mathcal C_k$ contain $e_{k+1}$,
we let $\mathcal{C}_{k+1} = 
\{C \in \mathcal{C}_k : e_{k+1} \in C\}$, and $X_{k+1} = X_k \cup \{e_{k+1}\}$.  
Otherwise we set $\mathcal{C}_{k+1} = 
\{C \in \mathcal{C}_k : e_{k+1} \notin C\}$ and $X_{k+1} = X_k$.  
Let $X = \bigcup_1^\infty X_i$.  Note that $\mathcal C_k$ is always 
an infinite set, 
and for all $i$, $j$, $i <j$, if $e_i \in X_i$, then 
$e_i \in C$ for all circuits $C \in \mathcal{C}_j$.

We claim that the set $X$ is dependent in $M'$ .  
By Lemma~\ref{cir2}, if $X$ is independent then there is a co-circuit $D$ of $M'$
that meets $X$ in exactly one element.
As $D$ is finite, we may pick an integer $k$
such that $D\subseteq\{e_1,\ldots,e_k\}$.  
Choose any $C\in\mathcal C_k$.
Since 
$C \cap \{e_1, e_2, e_3, \dots, e_k\} = X_k$, we see that $C$ also intersects $D$ in exactly one 
element, a contradiction to Lemma~\ref{cir2}.  
Thus, $X$ is dependent and therefore contains a circuit $C'$.
As $M$ is finitary, 
$C'$ contains a finite number of elements, and so $C' \subseteq X_\ell$ for 
some integer $\ell$.  
However, $\mathcal{C}_\ell$ contains an infinite number of circuits, each containing the set $X_\ell$.  
It follows that some circuit strictly contains $C'$, a contradiction.  
\end{proof}

\begin{proof} [Proof:  Theorem \ref{thm:fin}]
If we let $\mathcal{C}(M)$ be the set of circuits of a matroid $M$, 
an immediate consequence of the definition of the direct sum is that $\mathcal{C}\left( 
\bigoplus_{i \in I } M_i\right) = \bigcup_{i \in I}\mathcal{C}(M_i)$.
Moreover, the dual version of Lemma~\ref{cir2}
shows that every co-circuit is completely contained in some $M_i$.
It follows that if $M = \bigoplus_{i \in I}M_i$ where $M_i$ is finite for all $i \in I$, then $M$ 
is both finitary and co-finitary.  

To prove the other direction of the claim, let $M$ be a matroid 
that is both finitary and co-finitary.  
Let $M_i$ for $i \in I$ be the restriction of $M$ to the connected 
components.  For every $i \in I$, the matroid $M_i$ is connected, 
so by our assumptions on $M$ and by Lemma \ref{lem:fin2}, $M_i$ 
must be a finite matroid.  Lemma \ref{lem:fin1} implies that 
$M = \bigoplus_{i \in I}M_i$, and the theorem is proved.
 
\end{proof}

\section{Higher connectivity}\label{sec:highcon}

Let us recapitulate the definition of $k$-connectivity in finite 
matroids and see what we can keep of that in infinite matroids.
Let $M$ be a finite matroid on a ground set $E$. 
If $r_M$ denotes the rank function then 
the \emph{connectivity function}
$\cf$ is defined as 
\begin{equation}
\label{connfun}
\cf_M(X):=r_M(X)+r_M(E\sm X)-r_M(E)\text{ for }X\subseteq E.
\end{equation}
(We note that some authors define a connectivity function 
$\lambda$ by $\lambda(X)=k(X)+1$. In dropping the $+1$ we 
follow Oxley~\cite{OxleyBook92}.)
We call
a partition $(X,Y)$ of $E$ a \emph{$k$-separation} if 
$\cf_M(X)\leq k-1$ and $|X|,|Y|\geq k$.
The matroid $M$ is \emph{$k$-connected} if there exists no $\ell$-separation
with $\ell<k$.

Of these notions only the connectivity function is obviously useless
in an infinite matroid, as
the involved ranks will usually be infinite. We shall therefore 
only redefine $\cf$ and leave the other definitions unchanged. 
For this we have two aims. First, the new $\cf$ should coincide
with the ordinary connectivity function if the matroid is finite.
Second, $\cf$ should be consistent with connectivity as defined in
the previous section.

Our goal is to find a rank-free formulation of~\eqref{connfun}.
Observe that~\eqref{connfun} can be interpreted as the number
of elements we need to delete from the union of a 
basis of $M|X$ and a basis of $M|(E\sm X)$ in order to obtain a 
basis of the whole matroid. To show that this number does not depend
on the choice of bases is the main purpose of the next lemma.

\medskip

Let $M=(E,\mathcal I)$ be a matroid, and and let 
$I,J$ be two independent sets. We define
\[
\del_M(I,J):=\min\{ |F|:F\subseteq I\cup J,\, 
(I\cup J)\sm F\in\mathcal I\},
\]
where we set $\del_M(I,J)=\infty$ if there is no such finite set $F$.
Thus, $\del_M(I,J)$ is either a non-negative integer or infinity.
If there is no chance of confusion,
we will simply write $\del(I,J)$ rather than 
$\del_M(I,J)$

\begin{lemma}
\label{lem:del}
Let $M=(E,\mathcal I)$ be a matroid, let $(X,Y)$ be 
a partition of  $E$,
and let $B_X$ be a basis of $M|X$ and $B_Y$ a basis of $M|Y$.
Then
\begin{enumerate}[\rm (i)]
\item $\del(B_X,B_Y)=|F|$ for any $F\subseteq B_X\cup B_Y$
so that $(B_X\cup B_Y)\sm F$ is a basis of~$M$.
\item $\del(B_X,B_Y)=|F|$ for any $F\subseteq B_X$
so that $(B_X\sm F)\cup B_Y$ is a basis of~$M$.
\item $\del(B_X,B_Y)=\del(B'_X,B'_Y)$ for 
every basis $B'_X$ of $M|X$ and basis $B'_Y$ of $M|Y$.
\end{enumerate}
\end{lemma}

\begin{proof}
Let us first prove that 
\begin{equation}\label{eq1}
\begin{minipage}[c]{0.8\textwidth}\em
if for $F_1,F_2\subseteq B_X\cup B_Y$ it holds that $B_i:=(B_X\cup B_Y)\sm F_i$ is 
a basis of $M$ ($i=1,2$), then $|F_1|=|F_2|$.
\end{minipage}\ignorespacesafterend 
\end{equation} 
We may assume that one of $|F_1|$ and $|F_2|$ is finite, say $|F_2|$.
Then as $|B_1\sm B_2|= |F_2\sm F_1|<\infty$, it follows from Lemma~\ref{lem:basediff}
that $|F_1\sm F_2|=|F_2\sm F_1|$, and hence $|F_1|=|F_2|$.

(i) Let $F'\subseteq B_X\cup B_Y$ have minimal cardinality so that $(B_X\cup B_Y)\sm F'$
is independent. If $|F'|=\infty$ then, evidently, $F$ as in (i) needs to be an infinite set, too.
On the other hand, if $|F'|<\infty$ then $F'$ is also $\subseteq$-minimal, and thus
$(B_X\cup B_Y)\sm F'$ is maximally independent in $B_X\cup B_Y$,
and hence a basis of $M$. Now, $|F|=|F'|$ follows with~\eqref{eq1}.

(ii) Follows directly from~\eqref{eq1} and (i). 

(iii) Let $F\subseteq B_X$ as in (ii), i.e.\ $|F|=\del(B_X,B_Y)$. 
Because of the equivalence of (ii) and (iii)
in Lemma~\ref{lem:contrbase}, we obtain that $(B_X\sm F)\cup B'_Y$ is a 
basis of $M$ as well, and it follows that $\del(B_X,B_Y)=|F|=\del(B_X,B'_Y)$.
By exchanging the roles of $X$ and $Y$ we get then that 
$\del(B_X,B'_Y)=\del(B'_X,B'_Y)$, which finishes the proof.
\end{proof}

We now give a rank-free definition of the \emph{connectivity function}.
Let $X$ be a subset of $E(M)$ for some matroid $M$.
We pick an arbitrary basis $B$ of $M|X$, and a basis $B'$ of $M-X$
and define $\cf_M(X):=\del_M(B,B')$.
Lemma~\ref{lem:del}~(iii) ensures that $\cf$ is well defined,
i.e.\ that the value of $\cf_M(X)$ only depends on $X$ (and $M$)
and not on the choice of the bases.
The next two propositions demonstrate that $\cf$ extends
the connectivity function of a finite matroid and, furthermore, 
is consistent with connectivity defined in terms of circuits.

\begin{lemma}
If $M$ is a finite matroid on ground set $E$, and if $X\subseteq E$ then
\[
r(X)+r(E\sm X)-r(E)=\cf(X).
\]
\end{lemma}
\begin{proof}
Let $B$ be a basis of $M|X$, $B'$ a basis of $M-X$, and choose $F\subseteq B\cup B'$
so that $(B\cup B')\sm F$ is a basis of $M$. Then
\begin{eqnarray*}
\cf(X)&=&\del(B,B')=|F|
=|B|+|B'|-|(B\cup B')\sm F|\\
&=&r(X)+r(E\sm X)-r(E).
\end{eqnarray*}
\end{proof}

\begin{lemma}
A matroid is $2$-connected if and only if it is connected.
\end{lemma}
\begin{proof}
Let $M=(E,\mathcal I)$ be a matroid.
First, assume that there is a $1$-separation $(X,Y)$ of $M$. We need
to show that $M$ cannot be connected.
Pick $x\in X$
and $y\in Y$ and suppose there is a circuit $C$ containing both, $x$ and $y$. 
Then $C\cap X$ as well as $C\cap Y$ is independent, and so there 
are bases $B_X\supseteq C\cap X$ of $M|X$ and $B_Y\supseteq C\cap Y$ of
$M|Y$, by~(IM). As $(X,Y)$ is a $1$-separation, $B_X\cup B_Y$
needs to be a basis. On the other hand, we have $C\subseteq B_X\cup B_Y$, 
a contradiction.

Conversely, assume $M$ to be $2$-connected and pick a $x\in E$.
Define $X$ to be the set of all $x'$ so that $x'$ lies in a 
common circuit with $x$. If $X=E$ then we are done, with Lemma~\ref{lem:equiv}.
So suppose that $Y:=E\sm X$ is not empty, and choose a basis $B_X$
of $M|X$ and a basis $B_Y$ of $M|Y$. Since there
are no $1$-separations of $M$, $B_X\cup B_Y$ is dependent
and thus contains a circuit $C$. But then $C$ must meet $X$ as well as $Y$,
yielding together with Lemma~\ref{lem:equiv} a contradiction to
the definition of $X$.
\end{proof}

To illustrate the definition of $\cf$ and since it is relevant 
for the open problem stated below let us consider an example.
Let $T_\infty$ be the $\omega$-regular infinite tree,
and call any edge set in $T_\infty$ \emph{independent} if 
it does not contain a double ray (a $2$-way infinite path). 
The independent sets form a matroid $MT_\infty$~\cite{infaxioms}.
(It is, in fact, not hard to directly verify the independence 
axioms.)

What is the connectivity of $MT_\infty$?
Since every two edges of $T_\infty$ are contained in a common
double ray, we see that $M$ is $2$-connected. On the other hand, 
$M$ contains a $2$-separation: deleting an edge $e$ splits 
the graph $T_\infty$
into two components $K_1,K_2$. Put $X:=E(K_1)+e$ and $Y:=E(K_2)$,
and pick a basis $B_X$ of $M|X$ and a basis $B_Y$ of $M|Y$.
Clearly, neither $B_X$ nor $B_Y$ contains a double ray, while
every double ray in $B_X\cup B_Y$ has to use $e$. Thus, 
$(B_X\cup B_Y)-e$ is a basis of $M$, and $(X,Y)$ therefore a $2$-separation.

\medskip
It is easy to construct matroids of connectivity $k$ for 
arbitrary positive integers $k$. Moreover, there are 
matroids that have infinite connectivity, namely the
uniform matroid $U_{r,k}$ where $k\simeq r/2$. 
However, it can be argued that these matroids are simply too 
small for their high connectivity, and therefore more a 
fluke of the definition than a true example of an
infinitely connected matroid. Such a matroid should certainly
have an infinite ground set. 

\begin{problem}
Find an infinite infinitely connected matroid.
\end{problem}
As for finite matroids the minimal size of a circuit or
co-circuit is an upper bound on the connectivity. So, an infinitely
connected infinite matroid cannot have finite circuits or co-circuits.
In the matroid $MT_\infty$ above all circuits and co-circuits are infinite.
Nevertheless, $MT_\infty$ fails to be $3$-connected.

\section{Properties of the connectivity function}\label{sec:propcon}

In this section we prove a number of lemmas that 
will be necessary in extending Tutte's linking theorem to finitary matroids. As 
a by-product we will see that a number of standard properties of the connectivity
function of a finite matroid extend to infinite matroids, 
see specifically Lemmas~\ref{prop:kapdual},~\ref{lem:submod}, and~\ref{lem:minorconn}. 

Let us start by showing that connectivity is invariant under duality.
\begin{lemma}
\label{prop:kapdual}
For any matroid $M$ and any $X\subseteq E(M)$
it holds that $\cf_M(X)=\cf_{M^*}(X)$.
\end{lemma}
\begin{proof}
Set $Y:=E(M)\sm X$, 
let $B_X$ be a basis of $M|X$, and $B_Y$ a basis of $M|Y$.
By~(IM), we can pick $F_X\subseteq B_X$ and $F_Y\subseteq B_Y$ so that $(B_X\sm F_X)\cup B_Y$
and $B_X\cup(B_Y\sm F_Y)$ are bases of $M$.

Then $B^*_X:=(X\sm B_X)\cup F_X$ and $B^*_Y:=(Y\sm B_Y)\cup F_Y$
are bases of $M^*|X$ and $M^*|Y$, respectively. Indeed, 
$B_X\sm F_X$ is a basis of $M.X$, by Lemma~\ref{lem:contrbase},
which implies that $X\sm (B_X\sm F_X)=B^*_X$ is a basis 
of $(M.X)^*=M^*|X$. 
For $B^*_Y$ we reason in a similar way.

Moreover, since $B_X\cup (B_Y\sm F_Y)$ is a basis of $M$ 
we see from 
\[
(B^*_X\sm F_X)\cup B^*_Y=(X\sm B_X)\cup (Y\sm B_Y)\cup F_Y
= E(M) \sm (B_X\cup (B_Y\sm F_Y))
\]
that $(B^*_X\sm F_X)\cup B^*_Y$ is a basis of $M^*$.
Therefore
\[
\del_{M}(B_X,B_Y)=|F_X|=\del_{M^*}(B^*_X,B^*_Y),
\]
and thus $\cf_M(X)=\cf_{M^*}(X)$, as desired.
\end{proof}

\begin{lemma}
\label{lem:submod}
The connectivity function $\cf$ is submodular, i.e.\ for all
$X,Y\subseteq E(M)$ for a matroid $M$ it holds  that 
\[
\cf(X)+\cf(Y)\geq\cf(X\cup Y)+\cf(X\cap Y). 
\]
\end{lemma}
\begin{proof}
Denote the ground set of $M$ by $E$.
Choose a basis $B_\cap$ of $M|(X\cap Y)$, and a basis $B_{\overline\cap}$
of $M- (X\cup Y)$. Pick $F\subseteq B_\cap\cup B_{\overline\cap}$
so that $I:=(B_\cap\cup B_{\overline\cap})\sm F$ is a basis
of $M|(X\cap Y)\cup(E\sm (X\cup Y))$.
Next, we use (IM) to extend $I$ into $(X\sm Y)\cup (Y\sm X)$: 
let $I_{X\sm Y}\subseteq X\sm Y$
and $I_{Y\sm X}\subseteq Y\sm X$ so that $I\cup I_{X\sm Y}\cup I_{Y\sm X}$
is a basis of $M$.

We claim that $I_{\cup}:=B_\cap\cup I_{X\sm Y}\cup I_{Y\sm X}$ (and by symmetry
also $I_{\overline\cup}:=B_{\overline\cap}\cup I_{X\sm Y}\cup I_{Y\sm X}$), is independent.
Suppose that $I_\cup$ contains a circuit $C$. 
For each $x\in F\cap C$, denote by $C_x$
the (fundamental) circuit in $I\cup\{x\}$. 
As $C$ meets $I_{X\sm Y}\cup I_{Y\sm X}$, we have
$C\nsubseteq \bigcup_{x\in F\cap C}C_x$. Thus, (C3) is applicable
and yields a circuit $C'\subseteq (C\cup\bigcup_{x\in F\cap C}C_x)\sm F$.
As therefore $C'$ is a subset of the independent set $I\cup I_{X\sm Y}\cup I_{Y\sm X}$,
we obtain a contradiction.

Since $I_\cup$ is independent and $B_\cap\subseteq I_\cup$ a basis of $M|(X\cap Y)$, 
we can pick $F_\cup^X\subseteq X\sm (Y\cup I_{\cup})$ and
$F_\cup^Y\subseteq Y\sm(X\cup I_{\cup})$ so that $I_\cup\cup F_\cup^X\cup F_\cup^Y$ 
is a basis of $M|(X\cup Y)$.
In a symmetric way, we pick  $F_{\overline\cup}^X\subseteq X\sm (Y\cup I_{\overline\cup})$ and
$F_{\overline\cup}^Y\subseteq Y\sm(X\cup I_{\overline\cup})$
so that $I_{\overline\cup}\cup F_{\overline\cup}^X\cup F_{\overline\cup}^Y$
is a basis of $M-(X\cap Y)$.

Let us compute a lower bound for $\cf(X)$. Both sets 
$I_X:=B_\cap\cup I_{X\sm Y}\cup F_\cup^X$ and 
$I_{\overline X}:=B_{\overline\cap}\cup I_{Y\sm X}\cup F_{\overline\cup}^Y$
are independent. As furthermore $I_X\subseteq X$ and $I_{\overline X}\subseteq E\sm X$,
we obtain that $\cf(X)\geq\del(I_X,I_{\overline X})$. Since each 
$x\in F\cup F_\cup^X\cup F_{\overline\cup}^Y$ gives rise to a circuit
in $I\cup I_{X\sm Y}\cup I_{Y\sm X}\cup\{x\}$, we get that 
$\del(I_X,I_{\overline X})\geq |F|+|F_\cup^X|+|F_{\overline\cup}^Y|$.
In a similar way we obtain a lower bound for $\cf(Y)$. Together
these result in 
\[
\cf(X)+\cf(Y)\geq 
2|F|+|F_\cup^X|+|F_\cup^Y|+|F_{\overline\cup}^X|+|F_{\overline\cup}^Y|.
\]

To conclude the proof we compute upper bounds
for $\cf(X\cap Y)$ and $\cf(X\cup Y)$. Since
$B_\cap$ is  a basis of $M|(X\cap Y)$ and 
$B_{\overline\cup}:=I_{\overline\cup}\cup F_{\overline\cup}^X\cup F_{\overline\cup}^Y$
is one of $M-(X\cap Y)$, it holds that $\cf(X\cap Y)=\del(B_\cap,B_{\overline\cup})$.
Since $I\cup I_{X\sm Y}\cup I_{Y\sm X}$ is independent, we get that
$\del(B_\cap,B_{\overline\cup})\leq |F|+|F_{\overline\cup}^X|+ |F_{\overline\cup}^Y|$.
For $\cf(X\cup Y)$ the computation is similar, so that we obtain
\[
\cf(X\cap Y)+\cf(X\cup Y)\leq 
2|F|+|F_\cup^X|+|F_\cup^Y|+|F_{\overline\cup}^X|+|F_{\overline\cup}^Y|,
\]
as desired.
\end{proof}

\begin{lemma}
\label{lem:cap}
In a matroid $M$
let $(X_i)_{i\in I}$ be a family of nested subsets of $E(M)$,
i.e.\ $X_i\subseteq X_j$ if $i\geq j$, and set $X:=\bigcap_{i\in I} X_i$. 
If $\cf(X_i)\leq k$
for all $i\in I$ then $\cf(X)\leq k$.
\end{lemma}
\begin{proof}
Set $Y_i:=E(M)\sm X_i$ for $i\in I$, $Y:=\bigcap_{i\in I}Y_i=Y_1$
and $Z:=E(M)\sm (X\cup Y)$. Pick bases $B_X$ of $M|X$ and $B_Y$ of $M|Y$.
Choose $I_Z\subseteq Z$ so that $B_Y\cup I_Z$ is a basis of $M|(Y\cup Z)$.
Moreover, as $\cf(X_1)\leq k$ 
there exists a finite set (of size $\leq k)$ $F\subseteq B_Y$ so that 
$B_X\cup (B_Y\sm F)$ is a basis of $M|(X\cup Y)$, and a (possibly infinite)
set $F_Z\subseteq I_Z$ so that  $B_X\cup (B_Y\sm F)\cup (I_Z\sm F_Z)$
is a basis of $M$.

Suppose that $k+1\leq\cf(X)=|F|+|F_Z|$. Then choose $j\in I$
large enough so that $|F|+|F_Z\cap Y_j|\geq k+1$. Use (IM) to extend
the independent subset $B_X\cup (I_Z\cap X_j)\sm F_Z$ of $X_j$
to a basis $B$ of $M|X_j$. The set $B_Y\cup (I_Z\cap Y_j)$
is independent too, and we may extend it to a basis $B'$ of $M|Y_j$.
As $B_X\cup B_Y\cup (I_Z\sm (F_Z\cap X_j))\subseteq B\cup B'$
we obtain with
\[
\cf(X_j)=\del(B,B')\geq |F|+|F_Z\cap Y_j|\geq k+1
\]
a contradiction.
\end{proof}

For disjoint sets $X,Y\subseteq E(M)$ define 
\[
\cf_M(X,Y):=\min\{\cf_M(U) : X\subseteq U \subseteq E(M)\sm Y\}.
\]
Again, we may drop the subscript $M$ if no confusion is likely.

\begin{lemma}
\label{lem:minorconn}
Let $M$ be a matroid, and let $N=M/C-D$ be a minor of $M$.
Let $X$ and $Y$ be disjoint subsets of $E(N)$.
Then $\cf_N(X,Y)\leq\cf_M(X,Y)$.
\end{lemma}
\begin{proof}
Let $U\subseteq E(M)$ be such that $X\subseteq U \subseteq E(M)\sm Y$
and $\cf_M(U)=\cf_M(X,Y)$.
First suppose that $N=M-D$ for $D\subseteq E(M)\sm (X\cup Y)$.
Pick a basis $B'_{U}$  of $M|(U\sm D)$
 and extend it to a basis $B_U$ 
of $M|U$. Define $B'_{W}$ and $B_W$ analogously  for $W:=E(M)\sm U$.
Let $F\subseteq B_U\cup B_W$ be such that $(B_U\cup B_W)\sm F$ is a basis
of $M$. Since $B'_{U}$ and $B'_{W}$ are bases of 
$M|(U\sm D)=N|(U\sm D)$
resp.\ of $N|(W\sm D)$, and since clearly 
$(B'_{U}\cup B'_{W})\sm (F\sm D)$ is independent
it follows that $\cf_N(X,Y)\leq\cf_N(U\sm D)\leq |F|=\cf_M(X,Y)$.

Next, assume that $N=M/C$ for some $C\subseteq E(M)\sm (X\cup Y)$. Then, using 
 Lemma~\ref{prop:kapdual} we obtain
\[
\cf_N(X,Y)=\cf_{(M^*-C)^*}(X,Y) = \cf_{M^*-C}(X,Y)\leq
\cf_{M^*}(X,Y)=\cf_M(X,Y).
\]
The lemma follows by first contracting $C$ and then deleting $D$.
\end{proof}

\begin{lemma}
\label{lem:grow}
In a matroid $M$
let $X,Y$ be two disjoint subsets of $E(M)$, and let $X'\subseteq X$ and
$Y'\subseteq Y$ be such that $\cf(X',Y')=k-1$. Then
$\cf(X,Y)\geq k$ if and only if there exist $x\in X$ and $y\in Y$
so that $\cf(X'+x,Y'+y)=k$.
\end{lemma}
\begin{proof}
Necessity is trivial. To prove sufficiency,
assume that there exist no $x,y$ as in the statement.
For $x\in X$ denote by $\mathcal U_x$ the sets $U$  
with $X'+x\subseteq U\subseteq E(M)\sm Y'$ and $\cf(U)=k-1$. 
By our assumption,
$\mathcal U_x\neq\emptyset$.
By Zorn's Lemma and Lemma~\ref{lem:cap} there exists an $\subseteq$-minimal element $U_x$
in $\mathcal U_x$. 

Suppose there is a $y\in Y\cap U_x$. Again by the assumption, we can find
a set $Z$ with $X'+x\subseteq Z\subseteq E(M)\sm (Y'+y)$ and
$\cf(Z)=k-1$. From Lemma~\ref{lem:submod} it follows that 
$\cf(U_x\cap Z)=k-1$, and thus that $U_x\cap Z$ is an element 
of $\mathcal U_x$. As $y\notin U_x\cap Z$ it is strictly smaller than $U_x$
and therefore a contradiction to the minimality of $U_x$. Hence, $U_x$ 
is disjoint from $Y$.

Next, let $\mathcal W$ be the set of sets $W$ 
with $Y\subseteq W\subseteq E(M)\sm X'$ and $\cf(W)=k-1$. 
As $E(M)\sm U_x\in\mathcal W$ for every $x\in X$,
$\mathcal W$ is non-empty and we can apply Zorn's Lemma and Lemma~\ref{lem:cap}
in order to find an $\subseteq$-minimal element $W'$ in $\mathcal W$.
Suppose that there is a $x\in X\cap W'$. But then Lemma~\ref{lem:submod}
shows that $W'\cap (E(M)\sm U_x)\in\mathcal W$, a contradiction to the minimality
of $W'$.

In conclusion, we have found that $Y\subseteq W\subseteq E(M)\sm X$ and $\cf(W)=k-1$, 
which implies $\cf(X,Y)\leq k-1$.  This contradiction proves the claim.
\end{proof}

\section{The linking theorem}

\newtheorem*{InfLinking}{Theorem~\ref{thm:linking}}

We prove our main theorem in this section, Tutte's linking theorem
for finitary (and co-finitary) matroids, which we restate here:
\begin{InfLinking}
Let $M$ be a finitary 
or co-finitary matroid, and let $X$ and $Y$ be two disjoint subsets of $E(M)$.  Then 
there exists a partition $(C, D)$ of $E(M) \sm (X \cup Y)$ such that $\cf_{M/C-D}(X, Y) = 
\cf_M(X, Y)$.  
\end{InfLinking}

A fact that is related to Tutte's linking theorem, but quite a bit
simpler to prove, is that
for every element $e$ of 
a finite $2$-connected matroid $M$, one of $M/e$ or $M-e$ is still $2$-connected.
This fact extends to infinite matroids in a straightforward manner. Yet, in an
infinite matroid it is seldom necessary to only delete or contract
a single element or even a finite set. Rather, to be useful we would
need that 
\begin{equation}\label{fail}
\begin{minipage}[c]{0.8\textwidth}\em
for any set $F\subseteq E(M)$ of a $2$-connected matroid
$M=(E,\mathcal I)$ there always exists a partition $(A,B)$ of $F$
so that $M/A-B$ is still $2$-connected.
\end{minipage}\ignorespacesafterend 
\end{equation} 

\begin{figure}[htbp]
  \centering
  \includegraphics[width=0.7\linewidth]{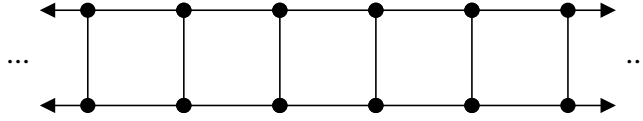}
  \caption{The double ladder}
  \label{fig:dladder}
\end{figure}
Unfortunately, such a partition of $F$ does not need to exist.
Indeed, consider the finite-cycle matroid $M_\text{FC}$ obtained from the
double ladder (see Figure~\ref{fig:dladder}), i.e.\ the matroid
on the edge set of the double in which an edge set is independent
if and only if it does not contain a finite cycle.
If $F$ is the set of rungs then 
we cannot contract any element in $F$ without destroying $2$-connectivity,
but if we delete all rungs we are left with two disjoint double rays.

In view of the failure of~\eqref{fail} in infinite matroids, even in finitary matroids 
like the example, 
it appears somewhat striking that Tutte's linking theorem does
extend to, at least, finitary matroids.

\medskip
Before we can finally prove Theorem~\ref{thm:linking} we need one more
definition and one lemma that will be essential when 
$\cf(X,Y)<\infty$.
Let $M'=M/C-D$ be a minor of a matroid $M$, for some disjoint
sets $C,D\subseteq E(M)$.  We say a $k$-separation $(X', Y')$ of $M'$ 
\emph{extends} to a $k$-separation of $M$ if there exists a  
$k$-separation $(X, Y)$ of $M$
such that $X' \subseteq X$, $Y' \subseteq Y$.  
The $k$-separation $(X', Y')$ is \emph{exact} if $\cf(X',Y')=k-1$.

\begin{lemma}
\label{lem:break}
Let $M$ be a matroid and let $X \cup Y \subseteq E(M)$ be disjoint subsets of $E(M)$.
Let $(X, Y)$ be an exact $k$-separation of $M|(X \cup Y)$ that does not 
extend to a  $k$-separation of $M$.  
Then there exist circuits $C_1$ and $C_2$ 
of $M$ such that 
$(X, Y)$ does not extend to a  $k$-separation of $M|(X \cup Y \cup C_1\cup C_2)$
\end{lemma}

\begin{proof}
We define 
\begin{equation*}
Comp_X:=\left \{ D : 
\text{ $D$ a component of $M/X$ such that $D \cap Y = \emptyset$}\right\}
\end{equation*} to be the 
set of connected components of $M/X$ that do not contain an element of $Y$.  Symmetrically, 
we define $Comp_Y$ to be the components of $M/Y$ that do not contain an element of 
$X$.

We claim that
\begin{equation}
\label{clm:subpart}
\emtext{
if $A \in Comp_X$ and $B \in Comp_Y$ such that $A \cap B 
\neq \emptyset$, then $A = B$.
}
\end{equation}

Assume the claim to be false and let $A$ and $B$ be a counterexample.  
Without loss of 
generality, we may assume there exists an element $x \in A \cap B$ and an element $y \in B\sm A$.  
By definition (and Lemma~\ref{cir3}), 
there exists a circuit $C_Y$ of $M$ such that $C_Y \sm  Y$ is a circuit of $M/Y$ containing $x$ and $y$.  
Now consider the circuit $C_Y$ in the matroid $M/X$.  
By Lemma~\ref{propB}, 
the dependent set $C_Y \sm  X$ (in fact, $C_Y$ is disjoint from $X$ but
we will not need this)
contains a circuit of $M/X$ that contains $x$  but not $y$
as $y$ and $x$ 
lie in distinct components of $M/X$.  
It follows that there exists 
a circuit $C_X$ in $M$ such that $x\in C_X\sm  X\subseteq C_Y-y$.  
By the finite circuit exchange axiom or (C3), there exists a circuit $C\subseteq C_X \cup C_Y$
of $M$ containing $y$ but not $x$. 
Hence, there is a circuit $D\subseteq C\sm  Y$ in $M/Y$ with $y\in D$ and $x\notin D$.
Since $y\in B$, $D$ cannot meet $X$, which implies 
$D\subseteq C\sm (X\cup Y)\subseteq C_Y\sm  Y$.
As $x\notin D$, the circuit $D$ in $M/Y$ is a strict subset of the circuit $C_Y\sm  Y$,
a contradiction.  
 This proves the claim.
 
Note that it is certainly possible that a set $A$ lies in both $Comp_X$ and $Comp_Y$. 
In a slight 
abuse of notation, we let $E(Comp_X) = \bigcup_{A \in Comp_X} A$, and 
similarly define $E(Comp_Y)$.  

Next, let us prove that
\begin{equation}\label{clm:extend}
\begin{minipage}[c]{0.8\textwidth}\em
 If 
 $E(Comp_X) \cup E(Comp_Y) \cup X \cup Y = E(M)$, then the separation $(X, Y)$ 
 extends to a $k$-separation of $M$.  
\end{minipage}\ignorespacesafterend 
\end{equation} 

Indeed, 
consider the partition $(L,R)$ for $L:=X\cup E(Comp_X)$ and 
$R:=E(M)\sm  L\supseteq Y$.
We claim 
that $(L,R)$ is a  $k$-separation of $M$.
Pick bases $B_X$ and $B_Y$ of $M|X$ resp.\ of $M|Y$,  
use (IM) to extend $B_X$ to a basis $B_L$ of $M|L$, and 
let $B_R$ be a basis of $M|R$ containing $B_Y$.  

Consider a circuit $C\subseteq B_L\cup B_R$ in $M$, 
and suppose $C$ to contain an element $x\in B_L \sm  X$.  
The set $C \sm  X$ contains a circuit $C'$ in  
$M/X$ containing  $x$.  Since $(C\cap B_L) \sm  X$ is 
independent in $M/X$, the circuit $C'$ must contain an element $y\in B_R$.  This 
implies that $x$ and $y$ are in the same component of $M/X$, and consequently, 
$y \in E(Comp_X)$.  This contradicts the definition of 
the partition, implying that no such circuit $C$ and element $x$ exist.  
A similar argument 
implies that $B_L \cup B_R$ does not contain any circuit 
containing an element of $B_R \sm  Y$ 
by considering $M/Y$.  
We conclude that every circuit contained in $B_L \cup B_R$ 
must lie in $B_X \cup B_Y$.  
As $\cf(X,Y)=k-1$, there exists a set of $k-1$ elements 
intersecting every circuit contained in $B_X \cup B_Y$, and thus in 
$B_L\cup B_R$, which implies that 
$(L,R)$ forms a $k$-separation.
This completes the proof of (\ref{clm:extend}).

Before finishing the proof of the lemma, we need one further claim.
\begin{equation}\label{clm:seps}
\begin{minipage}[c]{0.8\textwidth}\em
Let $C$ be a circuit of $M$ such that $C \sm  (X \cup Y)$ is a circuit of $M/(X \cup Y)$.  Then 
the only  $k$-separations of $M|(X \cup Y \cup C)$ that extend $(X, Y)$ 
are $(X \cup (C\sm  Y), Y)$ and 
$(X, Y \cup (C\sm  X))$.  
\end{minipage}\ignorespacesafterend 
\end{equation}

Assume that $(X', Y')$ is a  $k$-separation that extends $(X, Y)$ 
in the matroid $M|(X \cup Y \cup C)$.  Let $C' := C \sm  (X \cup Y)$.  Assume that 
$(X', Y')$ induces a proper partition of $C'$, i.e.\ that 
$C' \cap X' \neq \emptyset$ and $C'\cap Y' \neq \emptyset$.  
Picking bases $B_X$ of $M|X$ and $B_Y$ of $M|Y$ we observe that
$B_X \cup (C'\sm  Y')$ and $B_Y \cup (C' \sm  X')$ form bases of $M|X'$ and $M|Y'$ 
respectively.  Assume there exists a set $F$ of $k-1$ elements intersecting every circuit 
contained in $B_X \cup B_Y \cup C$.  By our assumption that $(X, Y)$ is an exact 
$k$-separation, 
we see that $F \subseteq B_X \cup B_Y$.  However, $C'$ is a circuit of $M/(X \cup Y)$, 
or, equivalently, $C'$ is a circuit of $M/((B_X \cup B_Y) \sm F)$.  It follows 
that there exists a circuit contained in $C' \cup B_X \cup B_Y$ avoiding the set $F$, a contradiction.  
This completes the proof of (\ref{clm:seps}).

\medskip

Since, by assumption, $(X,Y)$ does not extend to a $k$-separation of $M$
it follows from (\ref{clm:extend}) that there  
is an $e\notin E(Comp_X) \cup E(Comp_Y)$. Then
there exists a circuit $C_1$ of $M$ containing $e$ with $C_1 \cap Y \neq \emptyset$ 
such that the following hold:
\begin{itemize}
\item $C_1 \sm  X$ is a 
circuit of $M/X$, and 
\item $C_1 \sm  (X \cup Y)$ is a circuit of $M/(X \cup Y)$.  
\end{itemize} 
To see that such a circuit $C_1$ exists, recall first that $e\notin E(Comp_X)$
implies that there is a circuit $C_X$ in $M/X$ containing $e$ so that 
$C_X\cap Y\neq\emptyset$.
Then $C_X\sm  Y$ contains a circuit $C'$ in $M/(X\cup Y)$
with $e\in C'$ (see Lemma~\ref{propB}). 
For suitable $A_X\subseteq X$ and $A_Y\subseteq Y$
it therefore holds, by Lemma~\ref{cir3}, that $C'\cup A_X\cup A_Y$ is a circuit of $M$. 
If $A_Y=\emptyset$ then $C'$ would be a dependent set of $M/X$ 
strictly contained in the circuit $C_X$.
Thus, $A_Y\neq\emptyset$ and $C_1:=A_X \cup A_Y\cup C'$ has the 
desired properties. 
Symmetrically, there exists a circuit $C_2$ containing $e$ and intersecting 
$X$ in at least one element such that $C_2 \sm Y$ is a circuit of $M/Y$ and 
$C_2 \sm  (X \cup Y)$ 
is a circuit of $M/(X \cup Y)$.

Let us now see that the circuits $C_1$ and $C_2$ are as required 
in the statement of the lemma.  
To reach a contradiction, suppose that  $(X, Y)$ extends to a $k$-separation 
$(X', Y')$ of $M|(X \cup Y \cup C_1 \cup C_2)$.
By symmetry, we may assume that $e\in X'$. By (\ref{clm:seps}), we see 
that $C_1\sm (X \cup Y) \subseteq X'$ and $C_2\sm  (X \cup Y) \subseteq X'$
as well.  Pick a basis $B_X$ of $M|X$, and a basis $B_Y$ of $M|Y$.
From $C_1\cap Y\neq\emptyset$ it follows that  $C_1 \sm  Y$ is independent in $M/X$.
Thus, $B_X \cup (C_1 \sm  Y)$ is independent and we can extend it with (IM)
to a  basis $B_{X'}$ of $M|X'$.  The set $B_Y$ forms a basis 
of $M|Y'$ as $Y' =   Y$.  
Choose $F\subseteq B_{X'} \cup B_Y$ 
so that $(B_{X'} \cup B_Y) \sm  F$ is independent.
As $(X,Y)$ is an exact $k$-separation
and $(X',Y)$ thus too, it follows that $|F|=k-1$. 
As $\cf_{M|X \cup Y}(X, Y) = k-1$, we see
$F\subseteq B_X\cup B_Y$.
However, 
the set $C_1 \sm  (X \cup Y)$ is dependent 
in $M/(X \cup Y)$ and thus in $M/((B_X\cup B_Y)\sm F)$.
Hence, there is a set $S\subseteq (B_X\cup B_Y)\sm F$ so that
$(C_1\sm (X\cup Y))\cup S\subseteq B_{X'}\cup B_Y\sm F$ is dependent in $M$, 
contradicting our choice of $F$.  
This contradiction 
implies that the separation $(X, Y)$ does not extend to a 
$k$-separation of $M|(X \cup Y \cup C_1 \cup C_2)$, 
which concludes the proof of the lemma.
\end{proof}

We now proceed with the proof of the linking theorem for finitary matroids.

\begin{proof}[Proof of Theorem~\ref{thm:linking}]
By Lemma~\ref{prop:kapdual}, $\cf_{M}(Z)=\cf_{M^*}(Z)$
holds for any $Z\subseteq E(M)$, which means that 
we may assume $M$ to be finitary. Recall that this implies that every
circuit of $M$ is finite. 
We will consider the case when $\cf_M(X, Y) = \infty$ and 
$\cf_M(X, Y) < \infty$ separately.  

Assume that  $\cf_M(X, Y) = \infty$.
We inductively define a series of disjoint circuits $C_1, C_2, \dots$ 
in different minors of $M$ as follows. 
Starting with $C_1$ to be chosen as 
a circuit in $M$ intersecting  both $X$ and $Y$,
assume $C_1, \dots, C_t$ to be defined for $t \ge 1$.  
Note that as $Z:=\bigcup_{i=1}^t C_i$ is finite,
we still have  $\cf_{M/Z}(X \sm  Z, Y\sm Z) = \infty$.
Thus, there exists a circuit
$C_{t+1}$ in $M /Z$
that meets both both $X \sm  Z$ and 
$Y \sm  Z$. Having finished this construction, we let
$C_X = ( \bigcup_{i=1}^\infty C_i ) \cap X$, 
$C_Y = ( \bigcup_{i=1}^\infty C_i ) \cap Y$, and 
$C = ( \bigcup_{i=1}^\infty C_i)\sm  (X \cup Y)$.  
Set $D = E(M) \sm  (X \cup Y \cup C)$.

In order to show $\cf_{M/C-D}(X, Y) = \infty$
observe first that $C_X$ (and symmetrically, $C_Y$)  
is an independent set in $M/C$.  
If not then there exists a circuit $A \subseteq C_X \cup C$.  
Given that $M$ is finitary and $A$ thus finite, there 
exists a minimal index $t$ such that $A \subseteq \bigcup_{i=1}^t C_i$.  
It follows that $A \sm  (\bigcup_{i=1}^{t-1} C_i)$ 
is dependent in $M/(\bigcup_{i=1}^{t-1} C_i)$.
Since $A$ is disjoint from $Y$ but $C_t\cap Y\neq\emptyset$,
the dependent set $A \sm  (\bigcup_{i=1}^{t-1} C_i)$ 
is also a strict subset of the circuit $C_t$, 
a contradiction.

Let 
$B_X$ be a basis of $X$ containing $C_X$, and let $B_Y$
be a basis of $Y$ containing $C_Y$ in $M/C-D$.  
Assume there exists a finite set $F$ such that 
$(B_X \cup B_Y) \sm   F$ is a basis of $M/C-D$. 
Then for all $f \in F$, there exists a (fundamental)
circuit $A_f \subseteq B_X \cup B_Y$ of $M/C-D$ 
with $A_f \cap F = \{f\}$.  
Since the circuits $A_f$ are finite and the $C_i$ pairwise
disjoint, we may
choose $t$ large enough so that 
$C_t \cap A_f = \emptyset$ for all $f \in F$.
Lemma~\ref{propB} ensures the existence of 
a circuit $K\subseteq C_X\cup C_Y$ in $M/C-D$ with 
$K\sm \bigcup_{f\in F}A_f\neq\emptyset$.
By the finite circuit exchange axiom or (C3), there exists a circuit contained in 
$( K \cup \bigcup_{f \in F}A_f) \sm  F\subseteq (B_X\cup B_Y)\sm F$, 
a contradiction.  It follows that $\cf_{M/C-D}(X, Y) = \infty$, as claimed.

\medskip
We now consider the case when $\cf_M(X, Y) = k < \infty$.  By repeatedly 
applying Lemma~\ref{lem:grow}, there 
exists a set $X' \subseteq X$ and $Y' \subseteq Y$ such that $\cf_M(X', Y') = k$ and 
$|X'| = |Y'| = k$. (Observe that $\cf(X',Y')=k$ implies $|X'|,|Y'|\geq k$.) 
We shall find a partition $(C',D')$ of
$E(M)\sm (X'\cup Y')$ such that $\cf_{M/C' - D'}(X', Y') = k$.  
Then, Lemma~\ref{lem:minorconn} implies that
setting $C = C' \sm  (X \cup Y)$ and $D = D' \sm (X \cup Y)$ 
results in 
$\cf_{M/C-D}(X, Y) = k$ as desired. 

In order to find such $C'$ and $D'$ we 
will inductively define for $t\leq k$ finite
sets $Z_t \subseteq E(M)$ with $Z_{t-1}\subseteq Z_t$ 
such that in the restriction $M|Z_t$ it holds that 
$\cf_{M|Z_t} (X', Y') \geq t$.
For $t=1$, pick a circuit $A$   intersecting both $X'$ and $Y'$, 
and let $Z_1 =X' \cup Y' \cup A$. As $M$ is finitary $Z_1$ is finite,
and its choice ensures $\cf_{M|Z_1}(X', Y') \geq 1$.  

Assume that for $t<k$ 
we have defined $Z_{t-1}$, and and observe that as $Z_{t-1}$ is finite,
there are only finitely many $t$-separations in $M|Z_{t-1}$ 
separating $X'$ and $Y'$, all of which are exact.  
By applying Lemma~\ref{lem:break} to each of those, 
we deduce that there exists a finite set of circuits $A_1, A_2, \dots, A_l$ 
such that for $Z_t:=Z_{t-1} \cup A_1 \cup \dots \cup A_l$ we get
$\cf_{M|Z_t}(X', Y') \geq t$.  

To conclude, note that the matroid $M|Z_k$ is finite 
and that $\cf_{M|Z_k}(X', Y') = k$.  By Theorem~\ref{TLinking}, there 
exists a partition $(C', \tilde D)$ 
of $Z_k \sm  (X' \cup Y')$ such that $\cf_{(M|Z_k)/C'-\tilde D}(X', Y') = k$.  
Consequently, we obtain $\cf_{M/C'-D'}(X', Y') = k$
for $D':=\tilde D \cup (E(M) \sm  Z_k)$, which completes the proof of the theorem.  
\end{proof}

\bibliographystyle{amsplain}
\bibliography{../graphs}
\small
\vskip2mm plus 1fill
Version 15 March 2010

\end{document}